\documentclass[10pt]{article}
\usepackage[T2A]{fontenc}
\usepackage[cp1251]{inputenc}
\usepackage[english,russian]{babel}
\usepackage[tbtags]{amsmath}
\usepackage{amsfonts,amssymb,mathrsfs,amscd}
 \usepackage[hyper]{msb-a}
 \JournalName{}
\numberwithin{equation}{section}
\theoremstyle{plain}
\newtheorem{theorem}{Теорема}
\newtheorem{lemma}{Лемма}[section]
\newtheorem{propos}{Предложение}
\newtheorem{corollary}{Следствие}
\theoremstyle{definition}

\newtheorem{proof}{Доказательство}
\newtheorem{remark}{Замечание}

\begin{document}

\title{Экспоненциальный рост коразмерностей тождеств алгебр с единицей}
\author[M.\,V.~Zaicev]{М.\,В.~Зайцев}
\address{МГУ имени~М.\,В.~Ломоносова}
\email{zaicevmv@mail.ru}
\author[D.~D.~Repov\v s]{Д.~Д.~Реповш}
\address{Университет Любляны, Словения}
\email{dusan.repovs@guest.arnes.si}


\udk{}

\maketitle

\begin{fulltext}

\begin{abstract}
В работе изучается асимптотическое поведение экспоненциально ограниченных последовательностей
коразмерностей тождеств алгебр с единицей. Построена серия алгебр, у которых основание
экспоненты увеличивается ровно на $1$ при присоединении к исходной алгебре внешней
единицы. Показано, что PI-экспоненты унитарных алгебр могут принимать любое значение
больше двух, а экспоненты конечномерных унитарных алгебр образуют всюду плотное
подмножество в области $[2,\infty)$.
\end{abstract}
  

\markright{Рост коразмерностей алгебр с единицей}


\section{Введение}\label{s1}

\subsection{}\label{s1.1}
В статье изучаются функции, характеризующие количество тождественных соотношений,
выполняющихся в той или иной алгебре. Каждой алгебре $A$ над полем $F$ нулевой
характеристики можно сопоставить целочисленную последовательность $\{c_n(A)\}$,
$n=1,2,\ldots$, построенную по ее полилинейным тождествам. В асимптотическом поведении
этой последовательноти заложена определенная информация о строении самой алгебры
$A$. Например, если $A$ --- ассоциативная алгебра, то $c_n(A)=1$ для всех $n$
тогда и только тогда, когда $A$ --- коммутативная ненильпотентная алгебра. Если же
$c_n(A)=0$ для некоторого $n>1$, то $A$ нильпотентна, $A^n0$ (и наоборот). Недавно 
было показано, что $\{c_n(A)\}$ асимптотически возрастает, т.е. существует такое
натуральное $t$, что $c_{t+j}\le c_{t+j+1}$ для всех $j=0,1,\ldots$. Если при
этом $c_{t-1}> c_t$, то это значение $t$ тесно связано со ступенью нильпотентности 
радикала Джекобсона алгебры $A$ (результат анонсирован в \cite{GZ1}, полное
доказательство опубликовано в \cite{GZ2}). Если поле $F$ алгебраически замкнуто,
а $A$ проста, то $c_n(A) \sim d^n$, где $d=\dim A$ (\cite{R}. Здесь соотношение
$c_n(A) \sim d^n$ означает, что
$$
\lim_{n\to\infty} \sqrt[n]{c_n(A)}=d.
$$

Такой же эффект наблюдается и в случае алгебр Ли \cite{Z}, йордановых алгебр,
альтернативных алгебр и ряда других классов \cite{GSZ}. Для алгебр Ли хорошо известна 
открытая проблема классификации бесконечномерных простых алгебр Ли. В настоящее
время эта проблема, видимо, далека от своего решения, однако определенную
информацию о строении такой алгебры $L$ можно получить, если $\{c_n(L)\}$ имеет
экспоненциальный рост \cite{Rzm}.

\subsection{}\label{s1.2}
Наличие или отсутствие единицы в алгебре существенно скажывается на структуре ее
тождеств. Например, если $A$ --- ассоциативная алгебра с единицей, то совокупность
всех ее тождеств полностью определяется системой так называемых собственных
тождеств \cite{Sp}. Если, кроме того, $A$ удовлетворяет всем тождествам матричной 
алгебры $2\time 2$, то асимптотически для ее Т-идеала существует лишь счетное число
явно ипсываемых вариантов \cite{K}. Если $\{c_n(A)\}$ растет полиномиально, то
$c_n(A)= qn^k = O(n^{k-1}$ Для некоторого целого $k$ и положительного рационального
$q$ \cite{D}. Позднее было показано, что при фиксированном $k$ для любого
$q\in\mathbb{Q}, q>0$, можно подобрать подходящую алгебру \cite{DR}. И в той же работе
было доказано, что если $A$ --- унитарная алгебра, то
$$
\frac{1}{k!}\le q \le \sum_{i=2}^k \frac{(-1)^k}{i!} \simeq \frac{1}{e}.
$$

Еще один положительный эффект наличия единицы проявился в доказательстве следующей
гипотезы.  А.~Регев в качестве уточнения гипотезы Амицура предположил, что
$$
c_n(A)\simeq C n^\frac{t}{2}d^n
$$
для любой ассоциативной PI-алгебры, где $t$ и $n$ --- целые, $C=const$. После серии
частных результатов в 2008 г. гипотеза Регева была подтверждена для алгебр с 1
\cite{BR}, \cite{Be}. И только недавно был анонсирован результат о справедливости
этой гипотезы в общем случае \cite{GZ1}.

В работе \cite{GMZ} для всех вещественных $\gamma >1$ были построены примеры конечномерных
алгебр с экспоненциальным ростом коразмерностей $c_n\sim\gamma'<\gamma$. Как показано
в \cite{Z1},  для конечномерных алгебр с 1 экспонециальный рост не может быть медленнее
чем $2^n$.

В работе \cite{GZ3} было отмечено, что если $A$ --- ассоциативная PI-алгебра, а
$A^\#$ --- алгебра, полученная из $A$ путем присоединения внешней единицы, то
$exp(A^\#)=exp(A)$ или $exp(A)+1$. Это несложное утверждение вытекает из
результатов \cite{GZ4}, \cite{GZ5}, где не только было доказано существование 
предела
$$
exp(A)=\lim_{n\to \infty} \sqrt[n]{c_n(A)}
$$
для любой ассоциативной PI-алгебры $A$, но и предложена процедура вычисления этой
величины. Тем не менее, это наблюдение позволило выдвинуть гипотезу, что
$exp(A^\#$) всегда равняется $exp(A)$ или $exp(A)+1$. Первый нетривиальный пример,
подтверждающий эту гипотезу, был построен в \cite{Z1}, еще один пример предложен
в \cite{BBZ}, а в \cite{RZ} приведена уже серия примеров, в которых для любой алгебры
$A$ из работы \cite{GMZ} с $exp(A)=\gamma\in \mathbb R, 1\le\gamma\le 2$, ее
расширение $A^\#$ имеет экспоненту $exp(A^\#)=\gamma +1\in [2,3]$. Заметим также,
что в работе \cite{Ra1} автором была предложена конструкция построения по алгебре Ли $L$
над полем $F$ алгебры Пуассона, равной $A\oplus F$ как вектороне пространство и
содержит $L$ в качестве подалгебры Ли коразмерности 1. Алгебру $L\oplus F$
 можно считать естественной модификацией алгебры $L^\#$. Несколько позже тот же автор
 доказал, что $exp(L\oplus F)=exp(L)+1$ \cite{Ra2}.

\subsection{}\label{s1.3}
Основной целью данной работы является построение семейства алгебр $A_\gamma$, 
$\gamma\in\mathbb R$, $\gamma >1$, для которых $exp(A_\gamma)=\gamma$ (теорема \ref{t1}),
$exp(A_\gamma)^\#)=\gamma +1$ (теорема \ref{t2}). Отметим, что при построении 
этих примеров использовались бесконечные периодические слова и слова Штурма, 
комбинаторные свойства ктоторых использовались при получении асимптотических
оценок.

Кроме еще одного подтверждения уполянутой гипотезы, эти результаты показывают, что
любое вещественное число $\gamma \ge 2$ может быть реализовано как PI-экспонента
унитарной алгебры (следствие  \ref{c1}). Кроме того, из теоремы \ref{t2} и ряда 
комбинаторных свойств бесконечных слов следует, что PI-экспоненты конечномерных
унитарных алгебр образуют всюду плотное подмножество в области $[2,\infty)$.

С основами теории тождественных соотношений и количественной PI-теории можно
познакомиться в монографиях \cite{B}, \cite{Dren}, \cite{GZbook}.

\section{Основные понятия и конструкции}\label{s2}

\subsection{}\label{s2.1}
Пусть $A$ --- алгебра над полем $F$, а $F\{X\}$ --- абсолютно свободная $F$-агебра
с бесконечным множесвом порождающих $X$. Полином $$f=f(x_1,\ldots x_n,)\in F\{X\}, \ \ 
x_1,\ldots, x_n \in X,$$  называется тождеством $A$, если $f(a_1,\ldots, a_n)=0$
для любых $a_1,\ldots,a_n \in A$. Множество всех тождеств $Id(A)$ алгебры $A$ образует 
идеал в $F\{X\}$. Обозначим через $P_n$ подпространство всех полилинейных многочленов от
$x_1,\ldots, x_n$ в $F\{X\}$. Тогда $P_n\cap Id(A)$ --- множество всех полилинейных 
тождеств степени $n$ алгебры $A$. Хорошо известно, что в случае нулевой характеристики 
основного поля идеал $Id(A)$ полностью определяется набором подпространств
$\{P_n\cap Id(A)\}, n=1,2,\ldots $. Обозначим через $P_n(A)$ факторпространство
$$
P_n(A)=\frac{P_n}{P_n\cap Id(A)},
$$
а через $c_n(A)$ --- его размерность:
$$
c_n(A)=\dim P_n(A).
$$
Величина $c_n(A)$ называется $n$-й коразмерностью тождеств алгебры $A$ (или просто
$n$-й коразмерностью $A$) и является одной из количественных характеристик
совокупности тождественных соотношений лгебры $A$. Исследование асимптотического поведения 
последовательности $\{c_n(A)\}$ --- одна из ключевых задач количественной PI-теории.

В общем случае $\{c_n(A)\}$ может иметь сверхэкспоненциальный рость. Например,
если $A=F\{X\}$, то
$$
c_n(A)=\frac{1}{2}C_{2n-2}^{n-1} n!,
$$
если $A$ --- свободная ассоциативная алгебра, то $c_n(A)=n!$, а если $A$ ---
свободная алгебра Ли, то $c_n(A)=(n-1)!$. Однако во многих случаях рост последовательности 
$\{c_n(A)\}$ ограничен экспоненциальной функцией. Класс алгебр с экспоненциально ограниченным
ростом коразмерностей включает себя все ассоциативные PI-алгебры \cite{R1}, все 
конечномерные алгебры \cite{BD} любой сигнатуры, алгебры Каца-Муди \cite{Z2}, 
бесконечномерные простые алгебры Ли Картановского типа \cite{M} и целый ряд других.
В этом случае определены верхний и нижний пределы
$$
\overline{exp}(A) =\overline{\lim_{n\to\infty}}\sqrt[n]{c_n(A)}\, ,\quad
\underline{exp}(A)= \underline{\lim}_{n\to\infty}\sqrt[n]{c_n(A)},
$$
которые назыааются верхней и нижней PI-экспонентами $A$. Если существует обычный
предел, т.е. $\overline{exp}(A)=\underline{exp}(A)$, то его называют (обычной)
PI-экспонентой.

\subsection{}\label{s2.2}
При изучении асимптотики роста $\{c_n(A)\}$ полезным инструментом служит терия
представлений симметрических групп. Группа $S_n$ естественным образом действует на $P_n$:
$$
\sigma f(x_1,\ldots, x_n) = f(x_{\sigma(1)},\ldots, x_{\sigma(n)}).
$$
При этом подпространство $P_n\cap Id(A)$ инвариантно относительно этого действия,
и поэтому $P_n(A)$ также наделяется структурой $F[S_n]$-модуля. Все необходимые сведения
по теории представлений симметрических групп и ее применению при исследовании тождественных
соотношений можно найти в \cite{J}, \cite{B}, \cite{Dren}, \cite{GZbook}. В силу
полной приводимости представлений группы $S_n$  модуль $P_n(A)$ раскладывается в
прямую сумму неприводимых $F[S_n]$-модулей, что удобно записывать на языке теории
характеров. Характер $\chi(P_n(A))$ называется $n$-м кохарактером $A$ и обозначается как
$\chi_n(A)$. Разложение $P_n(A)$ на неприводимые компоненты записывается как
разложение $\chi_(A)$ в сумму неприводимых характеров:
\begin{equation}
\label{e1}
\chi_n(A)=\sum_{\lambda\vdash n}m_\lambda \chi_\lambda,
\end{equation}
где $\chi_\lambda$ --- характер неприводимого представления $S_n$, соответствующего
разбиению $\lambda$ числа $n$, а неотрицательное целое число $m_\lambda$ --- его
кратность. Соотношение (\ref{e1}) в частности означает, что
\begin{equation}
\label{e2}
c_n(A)=\sum_{\lambda\vdash n}m_\lambda d_\lambda,
\end{equation}
где $d_\lambda=\deg \chi_\lambda$ --- размерность неприводимого представления группы
$S_n$, соответствующего разбиению $\lambda$. Для получения оценок роста коразмерностей нам потребуется еще одна величина, называемая $n$-й кодлиной алгебры $A$, определяемая как
$$
l_n(A)=\sum_{\lambda\vdash n} m_\lambda,
$$
где $m_\lambda$ --- коэффициенты из правой части (\ref{e2}). Очевидно, что
\begin{equation}
\label{e2a}
c_n(A) \le l_n(A)\max\{d_\lambda\vert \lambda\vdash n, m_\lambda\ne 0\}.
\end{equation}

Нам потребуется более детальная информация о строении неприводимых $F[S_n]$-модулей.
Напомним, что разбиением $\lambda$ числа $n$ называется упорядоченный набор целых чисел
$\lambda=(\lambda_1,\ldots,\lambda_k)$, такой, что $\lambda_1\ge\ldots\ge\lambda_k>0$,
$\lambda_1+\ldots+\lambda_k=n$. Число $h(\lambda)=k$ называется высотой $\lambda$. По
разбиению $\lambda$ строится таблица из $n$ клеток, называемая диаграммой Юнга $D_\lambda$.
Она состоит из $k$ строк и содержит $\lambda_j$ клеток в $j$-й строке для каждого
$j=1,\ldots, k$. Если в клетки диаграммы $D_\lambda$ записаны числа $1,\ldots,n$, то
полученная конструкция называется таблицей Юнга $T_\lambda$. Известно, что любой 
неприводимый $F[S_n]$-модуль изоморфен минимальному левому идеалу $F[S_n]e_{T_\lambda}$
группового кольца группы $S_n$, где элемент $e_{T_\lambda}$ строится следующим образом.

Обозначим через $R_{T_\lambda}$ подгруппу всех подстановок, переставляющих числа
$1,\ldots, n$ только в пределах строк таблицы $T_\lambda$. Ясно, что
$R_{T_\lambda}\simeq S_{\lambda_1}\times\cdots\times S_{\lambda_k}$. Аналогично
определяется подгруппа $C_{T_\lambda}$, элементы которой не выводят каждое число за 
пределы столбца $T_\lambda$. Положим
$$
R(T_\lambda)=\sum_{\sigma\in R_{T_\lambda}} \sigma\, ,\quad
C(T_\lambda)=\sum_{\tau\in C_{T_\lambda}}({\rm sgn}\,\tau)\tau
$$
и
$$
e_{T_\lambda}= R(T_\lambda) C(T_\lambda).
$$
Характер этого модуля и называется неприводимым характером $\chi_\lambda$. Элемент 
$e_{T_\lambda}$ называется симметризатором Юнга и является квазиидемпотентом кольца
$F[S_n]$, т.е. $e_{T_\lambda}^2=\gamma e_{T_\lambda}$, где $\gamma$ --- ненулевой 
скаляр. Отсюда в частности следует, что элемент $C(T_\lambda)e_{T_\lambda}$ не равен 
нулю и порождает тот же самый минимальный левый идеал $F[S_n]e_{T_\lambda}$. В
контексте действия $S_n$ на пространстве полилинейных многочленов $P_n$ это
позволяет сделать несложный, но важный вывод.
\begin{remark}\label{r1}
Пусть $M$ --- неприводимый $F[S_n]$-подмодуль в $P_n$. Тогда $M$ порождается как
$F[S_n]$-модуль полилинейным ммногочленом со следующими свойствами:
\begin{itemize}
\item
множество переменных, входящих в $f$, распадается в объединение непересекающися
подмножеств
$$
\{x_1,\ldots,x_n\}=X_1\cup\ldots\cup X_t,
$$
где $t=\lambda_1$ --- длина первой строки $D_\lambda$, $|X_j|$ --- высота $j$-го столбца
$D_\lambda, j=1,\ldots,k$;
\item
полином $f$ кососимметричен по каждому из наборов $X_1,\ldots, X_t$.
\end{itemize}
\end{remark}

\subsection{}\label{s2.3}
Для оценок размерностей неприводимых представлений $S_n$ удобно пользоваться фукцией
$\Phi(\lambda)$, задаваемой на разбиениях следующим образом.

Пусть сначала $0\le x_1,\ldots, x_d\le 1$ --- любые вещественные числа, такие, что
$x_1,+\cdots+x_d=1$, а $d\ge 2$
\begin{equation}
\label{e3}
\Phi(x_1,\ldots, x_d)=\frac{1}{x_1^{x_1}\ldots x_d^{x_d}}.
\end{equation}
Мы будем пользоваться непрерывностью функции $\Phi$ и тем свойством, что если
зафиксировать значения всех переменных, кроме $x_i, x_j$, то максимум $\Phi$ достигается
при $x_i=x_j$. Более того, если  $x_i >x_j$, то $\Phi(x_i-\varepsilon, x_j+\varepsilon)$
растет с ростом $\varepsilon$ от $0$ до $\frac{1}{2}(x_i-x_j)$. Если же зафиксировать одну из переменных, например, $x_d=\gamma$, то максимум достигается при $x_1=\ldots=x_{d-1}$, т.е.
$$
\max\,\Phi=\Phi(\theta,\ldots,\theta,\gamma), \quad \hbox{где} 
\quad (d-1)\theta+\gamma=1. 
$$
Мы будем использовать обозначение
\begin{equation}
\label{e4}
\Phi_d(\theta)=\Phi(\underbrace{\theta,\ldots,\theta}_{d-1},\gamma),\quad (d-1)\theta+\gamma=1.
\end{equation}

Пусть теперь $\lambda=(\lambda_1,\ldots,\lambda_t)\vdash n$ и $d\ge t$. Мы будем записывать
$\lambda$ в виде $\lambda=(\lambda_1,\ldots,\lambda_d)$ даже если $t<d$, полагая
$\lambda_{t+1}=\ldots=\lambda_d=0$. Тогда
$$
\Phi(\lambda)=\Phi(\frac{\lambda_1}{n},\ldots,\frac{\lambda_d}{n}).
$$
Очевидно, что значение $\Phi(\lambda)$. Очевидно, что значение $\Phi(\lambda)$ не зависит 
от $d\ge t$, если использовать соглашение $0^0=1$.

Значение $\Phi(\lambda)$ и степень характера $d_\lambda=\deg\chi_\lambda$ связаны следующим
соотношением
\begin{lemma}\label{L1}\cite[лемма 1]{GZ6}
Пусть $\lambda=(\lambda_1,\ldots,\lambda_t)\vdash n$ --- разбиение $n$ на $t\le d$
компонент и $n\ge 100$. Тогда
$$
\frac{\Phi(\lambda)^n}{n^{d^2+d}}\le d_\lambda\le n\Phi(\lambda)^n.
$$
\end{lemma}

Нам потребуется следующее свойство функции $\Phi$. Пусть $\lambda=(\lambda_1,\ldots,\lambda_q)$, $\mu=(\mu_1,\ldots,\mu_q)$ --- два разбиения числа $n$, 
$\lambda_q, \mu+q >0$. Мы будем говорить, что диаграмма Юнга $D_\mu$ получена из диаграммы
 $D_\lambda$ выталкиванием вниз одной клетки, если существуют такие $1\le i<j \le q$,
 что $\mu_i= \lambda_i-1, \mu_j=\lambda_j+1$ и $\mu_k=\lambda_k$ для всех остальных
 $1\le k \le q$. Если же $\lambda=(\lambda_1,\ldots,\lambda_q)$, $\lambda_q>0$,
$\mu=(\mu_1,\ldots,\mu_q,1)\vdash n$, то $D_\mu$ получена выталкиванием вниз одной
клетки из $D_\lambda$, если одна из строк $D_\mu$ на одну клетку короче, чем у
$D_\lambda$,  а все остальные, кроме последней, имеют ту же длину.

\begin{lemma}\label{L2}\cite[лемма 3]{GZ6}, \cite[лемма 2]{ZR}
Пусть $D_\mu$ получена из $D_\lambda$ выталкиванием вниз одной клетки. Тогда
$\Phi(\mu) \ge \Phi(\lambda)$.
\end{lemma}

Мы также будем использовать и такое свойство функции $\Phi(x_1,\ldots,x_d)$.

\begin{lemma}\label{L3}\cite[лемма 2]{RZ}
Пусть $\Phi(x_1,\ldots,x_d)$ задана формулой (\ref{e3}) и пусть
$\Phi(z_1,\ldots,z_d)=a$ для некоторых фиксированных значений $z_1,\ldots,z_d$. Тогда
$$
\max_{0\le t \le 1}\{\Phi(y_1,\ldots,y_d,1-t)|y_1=tz_1,\ldots,y_d=tz_d \}=a+1,
$$
причем максимум достигается при $t=\frac{a}{a+1}$.
\end{lemma}

Лемма \ref{L3} фактически означает, что при добавлении к диаграмме $D_\lambda$ одной
дополнительной строки значение функции $\Phi(\lambda)$ увеличивается не более
чем на единицу.

 \subsection{}\label{s2.4}
 Для построения примеров алгебр с заданным характером поведения $\{c_n(A)\}$ мы
 воспользуемся подходом, впервые предложенном в работе \cite{GMZ} и базирующемся на комбинаторных свойствах бесконечных двоичных слов. Для этого напомним неекоторые
 понятия.
 
 Пусть $w=w_1w_2\ldots$ --- бесконечное слово в двоичном алфавите, т.е. все $w_i$ равны
 $0$ или $1$. Сложностью слова $w$ называется функция натурального аргумента $Comp_w(n)$,
 равная количеству различных подслов в $w$ длины $n$. Если слово $w$ периодическое, то
 $Comp_w(n)=const=T$ для всех $n\ge T$, где $T$ --- период $w$. Известно также, что если $w$
 не является периодическим, то $Comp_w(n) \ge n+1$ для всех $n\ge 1$ \cite{L}. Сумму
 $w_{k+1}+\ldots+w_{k+m}$ конечного подслова $u=w_{k+1}\ldots w_{k+m}$ принято обозначать как
 $h(u)$, а длину как $|u|$.
 
Для заданного слова $w$
\begin{equation} \label{e5}
\pi(w)=\lim_{n\to\infty} \frac{h(w_1,\ldots, w_n)}{n}
\end{equation}
называется наклоном слова $w$, если предел в правой части (\ref{e5}) существует.

Если $comp_w(n)=n+1$ для всех $n\ge 1$, т слово $w$ называется словом Штурма. Слова Штурма обладают следующими свойствами.

\begin{propos}\label{p1}
Пусть $w$ --- периодическое слово или слово Штурма. Тогда существует такая константа $C$, что
\begin{itemize}
\item[(1)]
$|h(x)-h(y)| \le C$ для любых конечных подслов $x$ и $y$ одинаковой длины;
\item[(2)]
наклон $\pi(w)$ всегда существует4
\item[(3)]
для любого конечного подслова $u$ в $w$
$$
\vert\frac{h(u)}{|u|} - \pi(w) \vert \le \frac{C}{|u|};
$$
\item[(4)]
для любого вещественного $\alpha\in (0;1)$ существует $w$ с $\pi(w)=\alpha$ и $w$ является
периодическим, если $\alpha$ --- рациональное число, либо словом Штурма, если $\alpha$ ---
иррациональное. Более того, можно взять $C=1$, если $w$ --- слово Штурма, либо $C=T$, если
$w$ --- периодическое слово с периодом $T$, и тогда
$$
\pi(w)=\frac{h(w_1\ldots w_T)}{T}.
$$
\end{itemize}
\end{propos}

В дальнейшем мы будем также считать слова из одних нулей или из дних единиц периодическими,
и тогда предложение \ref{p1} распространяется и на случаи $\alpha=0,\alpha=1$.

\section{Слова Штурма и неассоциативные алгебры}\label{s3}

В данном параграфе мы построим селейство неассоциативных алгебр,  PI-экс\-поненты которых 
принимают любые вещественные значения из области $[2;\infty)$. Идея потроения алгебр с 
заданным ростом коразмерностей на базе слов Штурма впервые была предложена и реализована 
в
 \cite{gmz2, GMZ}, где для любого вещественного $1\le \alpha \le 2$ юыла построена
алгебра $A_\alpha$ с $exp(A_\alpha)=\alpha$. В недавней работе \cite{RZ} было доказано,
что если к $A_\alpha$ присоединить внешнюю единицу, по у полученной алгебры 
$a_\alpha^\#$ экспонента существует и равна $\alpha+1$. Построенная ниже серия
алгебр обобщает конструкцию, предложенную в \cite{GMZ}. Следует отметить, что примеры алгебр
с произвольной PI-экспонентой $\alpha\ge 2$ также были приведены в \cite{GMZ}, однако
попытки их использования для построения унитарных алгебр с экспонентами болше трех не
привели к успеху. Это и вызвало необходимость построения новых примеров.

\subsection{}\label{s3.1}
Пусть $m$ и $d$ --- натуральные числа, $m\ge 2, d\le m-1$, и $w=w_1,w_2\ldots$ --- 
бесконечное слово в двоичном алфавите $\{0;1\}$. Рассмотрим бесконечную
последовательность $(m_1,m_2,\ldots)$, в которой $m_j=m+w_j$ для всех $j\ge 1$. 
 $A(m,d,w)$ задается своим базисом
$$
\{ a_i, b, z_{jk}^i|1\le i \le d, 1\le j \le m_k, k=1,2,\ldots\}
$$
и таблицей умножения
$$
z^i_{jk} a_i=
\left\{
  \begin{array}{rcl}
     z^i_{j+1,k}, &\quad \hbox{если} \quad & j<m_k  \\
    0, &\quad \hbox{если} \quad & j=m_k\, ,
           \end{array}
\right.
$$ 
$$
z^i_{m_k,k} b=
\left\{
  \begin{array}{rcl}
     z^{i+1}_{1k}, &\quad \hbox{если} \quad & i<d  \\
    z_{1,k+1}^1, &\quad \hbox{если} \quad & i=d\, .
           \end{array}
\right.
$$ 
Все остальные произведения базисных элементов равны нулю. Отметим некоторые свойства алгебры
$A(m,d,w)$;
\begin{itemize}
\item
алгебра $A(m,d,w)$ удовлетворяет тождеству $x_1(x_2x_3)\equiv 0$,
\item
линейная оболочка $<z^i_{jk}|1\le i\le d, 1\le j \le m_k, k \ge 1 >$ является идеалом в
$A(m,d,w)$ с нулевым умножением коразмерности $d+1$,
\item
если $f=f(x_1,\ldots, x_n)$ --- полилинейный многочлен степени $n\ge d+3$, кососимметричный 
по $x_1,\ldots, x_{d+3}$, то $f\equiv 0$ --- тождество в $A(m,d,w)$,
\item
если $f=f(x_1,\ldots, x_n)$ --- полилинейный многочлен степени $n\ge 2d+4$, кососимметричный 
по $x_1,\ldots, x_{d+2}$ и по $x_{d+3},\ldots, x_{2d+4}$, то $f\equiv 0$ --- тождество 
в $A(m,d,w)$.
\end{itemize}

Замечание \ref{r1} из предыдущего параграфа сразу же приводит к такому результату.

\begin{lemma}\label{L4}
Пусть $A(m,d,w)$ --- алгебра, заданная бесконечным словом $w$ и целочисленными параметрами
$m\ge 2$ и $1\le d \le m-1$. Если

\begin{equation}\label{e6}
\chi_n(A)=\sum_{\lambda\vdash n} m_\lambda \chi_\lambda
\end{equation}
--- $n$-й кохарактер алгебры $A$, то $m_\lambda\ne 0$ в (\ref{e6}) только при 
$h(\lambda)\le d+2$, где $h(\lambda)$ --- высота $\lambda$, т.е. число строк в диаграмме
$D_\lambda$. Кроме того, если $\lambda=(\lambda_1,\ldots,\lambda_{d+2})$ и $m_\lambda\ne 0$,
то $\lambda_{d+2}\le 1$.
\end{lemma}

\subsection{}\label{s3.2}
Для получения верхней оценки на рост $\{c_n(A(m,d,w)) \}$ нам необходимо сначала ограничить
рост кодлины $\{l_n(A(m,d,w)) \}$.

Пусть сначала $A$ --- произвольная алгебра. Обозначим через $R=R(y_1,y_2,\ldots)$
относительно свободную алгебру многообразия $var(A)$, порожденного алгеброй $A$, а через
$$
W_n^{(p)}(A)=Span\{y_{i_1}\ldots y_{i_n}|1\le i_1,\ldots,i_n\le p\}
$$
линейную оболочку всех одночленой степени $n$ от $y_{1},\ldots, y_{p}$ со всевозможными
расстановками скобок, т.е. всех однородных степени $n$ полиномов от $y_{1},\ldots, y_{p}$
в $R$.

\begin{lemma}\label{L5} \cite[лемма 4.1]{GMZ}
Пусть $A$ --- алгебра с $n$-м кохарактером 
$\chi_n(A)=\sum_{\lambda\vdash n} m_\lambda\chi_\lambda$. Тогда для любого $\lambda\vdash n$
с $h(\lambda)\le p$ выполняется неравенство
\begin{equation}\label{e7}
m_\lambda\le \dim W_n^{(p)}(A).
\end{equation}
\end{lemma}

Всюду  в дальнейшем мы будем опускать скобки в левонормированном произведении, т.е.
записывать $(zt)v$ как $ztv$. Это соглашение особенно удобно при работе с алгебрами
$A(m,d,w)$, поскольку все ненулевые произведения в них левонормированы в симу тождества 
$x_1(x_2x_3)\equiv 0$.

\begin{lemma}\label{L6} 
Пусть $A=A(m,d,w)$ задана $m,d$ и бесконечным словом $w$. Тогда
$$
\dim W_n^{d+2}(A) \le d(d+2)(m+1) Comp_w(n).
$$ 
\end{lemma}

\begin{proof}
Обозначим через $W$ линейную оболочку одночленов вида $ty_{i_1}\ldots y_{n-1}$, где
$t=y_{d+3}$, $1\le i_1,\ldots, y_{n-1}\le d+2$. Тогда
$$
\dim W_n^{(d+2)} \le (d+2)\dim W.
$$

Пусть $y$ --- некоторый элемент из $W$. Ясно, что $y$ --- ненулевой тогда и только тогда,
когда существует гомоморфизм $\sigma: R\to A$, при котором $\sigma(y)\ne 0$.

Чтобы получить оценку на размерность $W$ рассмотрим следующую конструкцию. Пусть
$F<a_1,\ldots, a_d,b>$ --- cсвободная ассоциативная алгебра с порождающими
$a_1,\ldots, a_d,b$ и $M$ --- свободный правый $F<a_1,\ldots, a_d,b>$-модуль с
одним порождающим $x$. Тогда любой элемент из $M$ можно записать в виде линейной комбинации элементов вида $xf(a_1,\ldots, a_d,b)$, где $f(a_1,\ldots, a_d,b)$ --- одночлен от
$a_1,\ldots, a_d,b$. 

Пусть теперь $\sigma$ --- гомоморфизм из $R$ в $A$. Ясно, что условие $\sigma(y)=0$, $y\in W$,
достаточно проверить только для всех гомоморфизмов вида
$$
\sigma(t)= z^i_{jk}, \sigma(y_s)=\alpha_1^sa_1+\cdots+\alpha_d^s a_d+\beta^s b,
1\le s \le d+2.
$$
В этом случае $\sigma$ можно преставить в виде композиции двух линейных отображений
$$
\psi: W \rightarrow M\quad {\rm and}\quad \varphi^i_{j,k}: M\rightarrow A,
$$
где
$$
\psi(ty_{i_1}\ldots y_{i_n-1})=
x(\alpha_1^{i_1}a_1+\ldots+\alpha_d^{i_1}a_d+\beta^{i_1} b)\ldots
(\alpha_1^{i_{n-1}}a_1+\ldots+\alpha_d^{i_{n-1}}a_d+\beta^{i_{n-1}} b),
$$
а
\begin{equation}\label{*e}
\varphi^i_{j,k}(xf(a_1,\ldots,a_d,b)=z^i_{jk}f(a_1,\ldots,a_d,b,
\end{equation}
где многочлен $f(a_1,\ldots,a_d,b$ в правой части (\ref{*e}) интерпретируется как многочлен 
от правых умножений на $a_1,\ldots,a_d,b$ в алгебре $A$.

Обозначим 
$$
I=\cap_{i,j,k} \ker\varphi^i_{j,k}.
$$
Тогда
$$
\dim W \le \dim\frac{M}{I}.
$$

Зафиксируем индексы $i,j,k$. Заметим сначала, что из правил умножения базисных элементов в 
$A$ следует, что существует ровно один одночлен $f^i_{j,k}$, не лежащий в ядре
 $\varphi^i_{j,k}$
$$
f^i_{j,k}=x\underbrace{a_i\ldots a_i}_{m_k-j}b 
\underbrace{a_{i+1}\ldots a_{i+1}}_{p_1}b\ldots b \underbrace{a_{i+r}\ldots a_{i+r}}_{p_r}b
\underbrace{a_{i+r+1}\ldots a_{i+r+1}}_{s},
$$
где индексы у $a_{i+r+1}\ldots a_{i+r+1}$ вычисляются по модулю $d$,
$m_k-j+p_1+\ldots+p_r+s+r+1=n-1$, $s\le d$, а все $p_1,\ldots, p_r$ равны одному из
$m_k,m_{k+1},\ldots$ и определяются однозначно подсловом $w(k, k+n-1) =
(w_k,w_{k+1},\ldots_{k+n-1})$ длины $n$ слова $w$. В частности, $f^i_{j,k}=f^i_{j,l}$ и
$\ker\varphi^i_{j,k}= \ker\varphi^i_{j,l}$, если $w(k, k+n-1)=w(l, l+n-1)$ в слове $w$.
Так как $1\le i \le d, 1\le j \le m+1$, то число различных ядер $\ker\varphi^i_{j,k}$ не
превосходит $d(m+1)Comp_w(n)$. Следовательно,
$$
\dim\frac{M}{I} \le d(m+1)Comp_w(n),\quad \dim W_n^{d+2}(A) \le d(d+2)(m+1) Comp_w(n),
$$
и лемма доказана.
\end{proof}

Практически точно так же доказывается и следующее утверждение, которое потребуется нам при присоединении единицы к исходной алгебре.

\begin{lemma}\label{L7}
Пусть $A=A(m,d,w)$ задана $m,d$ и бесконечным словом $w$. Тогда
$$
\dim W_n^{d+3}(A) \le d(d+3)(m+1) Comp_w(n).
$$ 
\end{lemma}

В качестве следствия мы получаем оценку роста кодлины для алгеьры, заданной словом Штурма или 
бесконечным периодическим словом.

\begin{propos}\label{p2}
Пусть $A=A(m,d,w)$, где $w$ ---  слово Штурма или бесконечное периодическое слово. 
Тогда
$$
l_n(A) \le 2d(d+2)(m+1)n^{d+1}(n+1).
$$
\end{propos}
\begin{proof}
Согласно лемме \ref{L4} мы имеем: $h(\lambda)\le d+2, \lambda_{d+2}\le 1$ для любого
разбиения $\lambda\vdash n$ с ненулевой кратностью $m_\lambda$. Количество таких
разбиений на превосходит $2d n^{d+1}$. Поэтому леммы \ref{L5} и \ref{L6}дают
требуемую оценку.
\end{proof}

\subsection{}\label{s3.3}
Теперь мы можем приступить к получению верхних оценок PI-экспонент.

Пусть $A=A(m,d,w)$ --- алгебра, построенная по бесконечному слову $w$ и пусть теперь
$w$ --- периодическое слово или слово Штурма. Если $$f=f(z^i_{jk},a_1,\ldots,a_d,b)$$
ассоциативное слово в алфавите $\{z^i_{jk},a_1,\ldots,a_d,b\}$, то можно говорить о его
степенях $\deg_b f,\deg_{a_i}f,\deg_{z^i_{jk}}$ по переменным, об общей степени $\deg f$,
а также о значении $f$ в $A$, если рассматривать его как левонормированное произведение
базисных элементов.

Нам понадобится одно достаточнон условие того, что $f\ne 0$.

\begin{lemma}\label{A0}
Для заданных $m,d,w$ найдется такая последовательность $\{\varepsilon_n>0\}, n=1,2,\ldots$, 
что если $f=f(z^i_{jk},a_1,\ldots,a_d,b)$ --- одночлен степени $n$, не равный нулю в
$A(m,d,w)$, то
$$
\frac{\deg_b F}{n} \le\frac{1}{m+\alpha}+\varepsilon_n,
$$
где $\alpha=\pi(w)$ --- наклон слова $w$. При этом $\varepsilon_n\to 0$, если $n\to\infty$.
\end{lemma}
\begin{proof}
Слово $f$ можно записать в виде $f=ZPQ$, где $Z$ --- произведение базисных элементов 
$\{z^i_{jk},a_\alpha, b\}$ степени $\deg Z\le (m+1)d, Q=Q(a_1,\ldots,a_d,b), \deg Q \le 
(m+1)d$, а 
$$
P=a_1^{m_k-1}b\ldots a_d^{m_k-1}b\ldots a_1^{m_{k+t-1}-1}b\ldots a_d^{m_{k+t-1}-1}b.
$$
Тогда $\deg_b P=td$ и
$$
\deg_{a_i}P=(m_k-1)+\ldots+(m_{k+t-1}-1)=m_1+\ldots+m_{k+t-1}-t=
(m-1)t+w_k+\ldots+w_{k+t-1}
$$
для любого $i=1,\ldots,d$. Как отмечено в предложении \ref{p1} для слова $w$ существует такая
 константа $C$, что $|w_k+\ldots+w_{k+t-1}-\alpha t| \le C$. Поэтому
$$
\deg P=dmt+d(w_k+\ldots+w_{k+t-1})\ge dt(m+\alpha-\frac{C}{t})
$$
и $n=\deg f\ge \deg P$, а $\deg_b f \le td+2d=(t+2)d$. Следовательно,
$$
\frac{\deg_b f}{n} \le \frac{1+\frac{2d}{t}}{m+\alpha-\frac{C}{t}}.
$$
Поскольку $n\le d(m_k+\ldots+m_{k+t-1})+2(m+1)d \le d(m+1)t+2(m+1)d$, то
$$
t\ge \frac{n}{d(m+1)} - \frac{2}{d}
$$
и $t$ растет линейно с ростом $n$. Следовательно,
$$
\lim_{n\to\infty} \frac{\deg_b f}{n} = \frac{1}{m+\alpha},
$$
откуда следует утверждение леммы.

\end{proof}

Теперь мы получим оцеку сверху на рост коразмерностей алгебры $A(m,d,w)$.

\begin{lemma}\label{A1}
Пусть $A=A(m,d,w)$, где $w$ --- бесконечное периодическое слово или  слово Штурма с
наклоном $\alpha=\pi(w)$. Тогда
$$
\overline{exp}(A)=\Phi_d(\frac{1}{m+\alpha}),
$$
где функция $\Phi_d$ задана формулой (\ref{e4}).
\end{lemma}
\begin{proof}
Зафиксируем произвольное малое $\varepsilon>0$ и покажем, что для него существует такое
$N$, что если $n\ge N, \lambda\vdash n$ и $m_\lambda\ne 0$ в (\ref{e6}), то
$$
\Phi(\lambda) \le \Phi_d\left(\frac{1}{m+\alpha}+\varepsilon\right).
$$

Пусть сначала $\lambda_{d+1}=0$, т.е. $\lambda=(\lambda_1,\ldots,\lambda_d,0,0)$. Тогда
$$
\Phi(\lambda) \le\Phi\left(\frac{1}{d},\ldots,\frac{1}{d},0,0\right) \le 
\Phi\left(\underbrace{\theta,\ldots,\theta}_d,\frac{1}{m+\alpha}\right) =
\Phi_d\left(\frac{1}{m+\alpha} \right).
$$

Пусть теперь $\lambda_{d+1}\ne 0$. Тогда в силу замечания \ref{r1} существует полилинейный
многочлен $h=h(x_1,\ldots,x_n)$ кососимметричный по $\lambda_1$ наборам переменных
$X_{1},\ldots, X_{\lambda_1}$, причем $|X_1|=d+1$ или $d+2$ в зависимости от значения
$\lambda_{d+2}$ (0 или 1), а $|X_2|=\ldots= |X_{\lambda_1}|=d+1$, не являющийся
тождеством $A$. Следовательно, существует такая подстановка $\varphi: X\to\{a_r,d,z^i_{jk}\}$,
что $f=\varphi(h)=f(z^i_{jk},a_1,\ldots, a_d,b)$ --- ненулевой одночлен в $A$. Тогда
$\deg_b f\ge \lambda_{d+1}$, и по лемме \ref{A0}
$$
\frac{\lambda_{d+1}}{n} \le \frac{\deg_b f}{n} \le \frac{1}{m+\alpha} + \varepsilon_n.
$$
Если $\lambda_{d+2}=0$, то
$$
\Phi(\lambda) \le \Phi(\underbrace{\theta,\ldots,\theta}_d,\frac{1}{m+\alpha}+\varepsilon_n,0)
= \Phi_d(\frac{1}{m+\alpha}+\varepsilon_n) \le \Phi_d(\frac{1}{m+\alpha}+\varepsilon)
$$
при всех больших $n$, поскольку $\varepsilon_n\to 0$  с ростом $n$, а функция
$\Phi_d(\frac{1}{m+\alpha}+x)$ возрастает  при увеличении $x$. Если же $\lambda_{d+2}=1$,
$$
\Phi(\lambda) \le 
\Phi(\underbrace{\theta,\ldots,\theta}_d,\frac{1}{m+\alpha}+\varepsilon_n,\frac{1}{n}).
$$
Поскольку $\varepsilon_n,\frac{1}{n}\to 0$ при $n\to\infty$, то найдется такое $n$, что
$$
\frac{1}{m+\alpha}+\varepsilon_n+\frac{1}{n} < \frac{1}{m+\alpha}+\varepsilon,
$$
$$
\max \left(\frac{1}{m+\alpha}+\varepsilon_n\right)^{(\frac{1}{m+\alpha}+\varepsilon_n)}
n^{-n} \le 
(\frac{1}{m+\alpha}+\varepsilon)^{(\frac{1}{m+\alpha}+\varepsilon)}.
$$
Следовательно,
$$
\Phi(\lambda) \le 
\Phi\left(\theta,\ldots,\theta,\frac{1}{m+\alpha}+\varepsilon_n,\frac{1}{n}\right) \le 
\Phi\left(\theta',\ldots,\theta',\frac{1}{m+\alpha}+\varepsilon_n,0\right)
$$
$$ 
= \Phi_d\left(\frac{1}{m+\alpha}+\varepsilon\right),
$$
где $\theta'd+ \frac{1}{m+\alpha}+\varepsilon=1$ и $\theta' \ge \theta$.
Поскольку
$$
c_n(A)=\sum m_\lambda d_\lambda \le l_n(A) \max\{d_\lambda| \lambda\vdash n, m_\lambda\ne 0\},
$$
то из леммы \ref{L1} и предложения \ref{p2} следует, что
$$
\overline{\lim_{n\to\infty}}\sqrt[n]{c_n(A)} \le 
\Phi_d\left(\frac{1}{m+\alpha}+\varepsilon\right)
$$
для любого фиксированного $\varepsilon >0$. Следоательно, 
$$
\overline{exp}(A) \le \Phi_d\left(\frac{1}{m+\alpha}\right),
$$
и лемма доказана.

\end{proof}

Теперь перейдем к нижней оценке роста коразмерностей алгебры $A(m,d,w)$.

\begin{lemma}\label{A2}
Пусть $A(m,d,w)$ --- алгебра из леммы \ref{A1}. Тогда $\underline{exp}(A)\ge
\Phi_d\left(\frac{1}{m+\alpha}\right)$, где $\alpha=\pi(w)$ --- наклон слова $w$.
\end{lemma}

\begin{proof}
Рассмотрим одночлен
$$
h_1=zx^1_1x^1_2\ldots x^1_dy^1_1\ldots x^d_1x^d_2\ldots x^d_py^1_d
$$
в свободной алгебре $F\{X\}$ степени $(p+1)d+1$, где $p=m_1-1 \ge m-1\ge d$. Пусть 
$Alt^1_1:P_{(p+1)d+1}\to P_{(p+1)d+1}$ --- оператор альтернирования по
$z,x^1_1, x^2_1,\ldots, x^d_1,y^1_1$, а $Alt^1_i$ --- оператор альтернирования по
$x^1_i, x^2_i,\ldots, x^d_i,y^1_i$ для всех $2\le i \le d$. Если $p>d$, то обозначим 
также через $Alt^1_{d+j}$ альтернирование по $x^1_{d+j}, x^2_{d+j},\ldots, x^d_{d+j}$
для всех $1\le j \le p-d$. Положим $f_1= Alt^1_1\ldots Alt^1_p(h_1)$.

Рассмотрим подстановку $\varphi: X\to A$, при которой
$$
\varphi(z)=z^1_{11},\varphi(x^1_1)=\ldots =\varphi(x^1_p)=a_1,\ldots,
\varphi(x^d_1)=\ldots =\varphi(x^d_p)=a_d,
$$
$$
\varphi(y^1_1)=\ldots =\varphi(y^1_d)=b.
$$
Тогда
$$
\varphi(f_1)= z^1_{11}\underbrace{a_1\ldots a_1}_{m_1}b\ldots
\underbrace{a_d\ldots a_d}_{m_1}b= z^1_{12}.
$$
Заметим, что результат подстановки $\varphi$ не изменится, если применить ее не к самому
элементу $f_1$, а к его симметризации $Sym\, f_1$, где $Sym$ означает симметризацию 
по наборам
$\{x^1_1,\ldots, x^1_p\}$, $\ldots$, $\{x^d_1,\ldots, x^d_p\}$, $\{^1_1,\ldots, y^1_d\}$. 
Тогда многочлен $Sym\, f_1$ порождает в $P_{(p+1)d+1}$ неприводимый $F[S_{(p+1)d+1}]$-модуль,
соответствующий разбиению $\lambda =(\lambda_1,\ldots,\lambda_{d+2})$, где
$\lambda_1=\ldots=\lambda_d=p=M_1$, $\lambda_{d+1}=d, \lambda_{d+2}=1$, а условие
$\varphi(Sym\, f_1)\ne 0$ означает, что кратность $m_\lambda$ в разложениии (\ref{e1}) не
равна нулю.

Обозначим $p_1=p$. Далее для $j=2,3,\ldots$ строим многочлены $f_2,f_3,\ldots$ следующим
образом. Если $f_1,\ldots, f_{j-1}$ уже построены, то берем
$$
h_j=f_{j-1}x^1_{q+1}\ldots x^1_{q+p_j}y^1_j\ldots
x^d_{q+1}\ldots x^d_{q+p_j}y^d_j,
$$
где $q=p_1+\ldots +p_{j-1},\, p_j=m_j-1$ и определяем $f_j$ как
$$
f_j=Alt^j_1\ldots Alt^j_{p_j}(h_j),
$$
где $Alt^j_1,\ldots ,Alt^j_d$ ---  альтернирования по наборам 
$\{x^1_{q+1},\ldots,x^d_{q+1}, y^j_1\}$, $\ldots$,\\ $\{x^1_{q+d},\ldots,x^d_{q+d}, y^j_d\}$ 
соответственно. Если же $p_j>d$, то $Alt^j_{d+i}$ --- альтернирование по
$\{x^1_{q+d+i},\ldots,x^d_{q+d+i}\}$, $1\le i \le p_j-d$. Расширим действие подстановки
$\varphi: X\to A$, построенной на $(j-1)$-м шагеб полагая
$$
\varphi(x^1_{q+1})=\ldots= \varphi(x^1_{q+p_j})=a_1,\ldots,
\varphi(x^d_{q+1})=\ldots= \varphi(x^d_{q+p_j})=a_d,
$$
$$
\varphi(y^j_{1})\ldots= \varphi(y^j_d)=b.
$$
Тогда, ка и прежде,
$$
\varphi(Sym\, f_j)=z^1_{1,j+1} \ne 0,
$$
где симетризация $Sym$ проводится по наборам
$$
\{x^1_1,x^1_2,\ldots,x^1_{q+p_j}\},\ldots,\{x^d_1,x^d,\ldots,x^d{q+p_j}\},
\{y^1_1,\ldots,y^1_d,\ldots,y^j_1,\ldots, y^j_d\}.
$$
Тогда, как и при $j=1$, $Sym\, f_j$ порождает неприводимый модуль с характером $\chi_\lambda$,
где $\lambda=(\lambda_1,\ldots,\lambda_{d+2})$, $\lambda_1=\ldots= \lambda_d=m_1+\ldots+m_j-j,
\lambda_{d+1}=jd, \lambda_{d+2}=1$, и $m_\lambda\ne 0$ в (\ref{e1}).

Таким образом, для каждого натурального $t$ мы построили не являющийся тождеством 
многочлен $f_t$ степени
$$
n=n(t)=(m_1+\ldots+m_t)d+1=tmd+d(w_1+\ldots+w_t)+1.
$$
При этом ненулевое значение $f_t$ принимает при подстановке $\varphi: X\to A$, когда
элемент $b$ подставляется $td$ раз. Тогда по лемме \ref{A0}
$$
\frac{td}{n} \le \frac{1}{m+\alpha}+\varepsilon_n,
$$
где $\alpha=\pi(w)$ --- наклон $w$, а $\varepsilon_n\to 0$  при $n\to\infty$. Кроме того,
симметризация $Sym\, f_t$ тоже не является тождеством в $A$, $\varphi(Sym\, f_t)=
K\cdot\varphi(f_t), K\ne 0$, и порождает в $P_n$ неприводимый $F[S_n]$-модуль с характером
$\chi_{\lambda^{(n)}}$, где
$$
\lambda^{(n)}=(\lambda_1,\ldots,\lambda_{d+2}),\,
\lambda_1=\ldots=\lambda_d,\, \lambda_{d+1}=td,\, \lambda_{d+2}=1.
$$
Следовательно,
$$
\frac{\lambda_{d+1}}{n}=\frac{1}{m+\frac{w_1+\ldots+w_t}{t}+\frac{1}{td}}=\beta
$$
и
$$
\Phi(\lambda^{(n)})=
\Phi\left(\frac{\lambda_1}{n},\ldots,\frac{\lambda_1}{n},\beta,\frac{1}{n}\right).
$$

Чтобы получить оценку снизу на $\Phi(\lambda^{(n)})$, воспользуемся свойствами периодических 
слов и слов Штурма. Согласно предложению \ref{p1}
$$
\lim_{n\to\infty} \frac{w_1+\cdots+w_t}{t}=\alpha,
$$
а поскольку $mtd \ le n \le (m+1)td$, то величину $\frac{w_1+\cdots+w_t}{t}$ можно сделать
сколь угодно близкой к $\alpha$ для всех достаточно больших $n$. Следовательно, для
любого $\varepsilon>0$ найдется такое $N$, что
$$
\beta=\frac{1}{m+\frac{w_1+\ldots+w_t}{t}+\frac{1}{td}} \ge\frac{1}{m+\alpha}-\varepsilon
$$
при всех $n\ge N$. Тогда из свойств функции $\Phi$ мы получаем
$$
\Phi(\lambda^{(n)})=
\Phi\left(\underbrace{\theta,\ldots,\theta}_{d?},\frac{1}{m+\alpha}-\varepsilon,0\right)=
\Phi\left(\frac{1}{m+\alpha}-\varepsilon\right),
$$
где $\theta d+ \frac{1}{m+\alpha}-\varepsilon=1$.

Так как 
$$
c_n(A)\ge d_{\lambda^{(n)}} \ge \frac{1}{n^{(d+2)^2+d+2}}\Phi\left(\lambda^{(n)} \right)^n
$$
в силу леммы \ref{L1}, а $\varepsilon >0$ выбрано произвольно, то
$$
\underline{\lim}_{n(t)\to\infty} \sqrt[n(t)]{c_{n(t)}(A)} \ge 
\Phi_d\left(\frac{1}{m+\alpha}\right).
$$
Осталось заметить, сто $c_n(A)$ --- неубывающая последовательность и что 
$n(t+1)-n(t) \le (m+1)d$, откуда следует равенство
$$
\underline{exp}(A)=\underline{\lim}_{n\to\infty} \sqrt[n]{c_{n(t)}(A)} \ge 
\Phi_d\left(\frac{1}{m+\alpha}\right).
$$
Лемма доказана.
\end{proof}

Леммы \ref{A1} и \ref{A2} сразу же дают нам основной результат данного параграфа.

\begin{theorem}\label{t1}
Пусть $m$ и $d$ --- целые числаб $m\ge 2, 1\le d \le m-1$, а $w$ --- бесконечное периодическое 
слово или слово Штурма с  наклоном $\alpha$. Тогда PI-экспонента алгебры $A=(m,d,w)$ существует и равна
$$
exp(A)=\Phi_d\left(\frac{1}{m+\alpha} \right)=
\Phi\left(\underbrace{\frac{m+\alpha+1}{d(m+\alpha)},\ldots,
\frac{m+\alpha+1}{d(m+\alpha)}}_d,\frac{1}{m+\alpha} \right).
$$
\end{theorem}

\section{Экспоненты алгебр с присоединенной единицей}\label{s4}

\subsection{}\label{s4.1}
Напомним, что если к алгебре $A$ присоединяется внешним образом единица, то полученную 
в результате алгебру мы обозначаем как $A^\#$. мы будем присоединять единицы к алгебрам
$A(m,d,w)$, рассмотренным в предыдущем параграфе.

Нам понадобится технический результат зи работы \cite{RZ}.

Напомним, что для заданной алгебры $B$ через $W_n^{(p)}(B)$ обозначается подпространство всех однородных степени $n$ многочленов от $y_1,\ldots, y_p$ в относительно  свободной алгебре
$R(y_1,y_2,\ldots)$ многообразия $var(B)$ со свободными порождающими $y_1,y_2,\ldots$~.

\begin{lemma}\label{LL1}\cite[лемма 6]{RZ}
Пусть $B$ --- произвольная алгебра и пусть $$\dim W_n^{(p)}(B) \le\alpha n^T$$ для некоторых
$\alpha\in\mathbb{R}$ и $T\in \mathbb{N}$. Тогда 
$\dim W_n^{(p)}(B^\#) \le\alpha(n+1)^{T+p+1}$ 
\end{lemma}

Сначала мы оценим сверху рост кодлины.

\begin{lemma}\label{LL2}
Пусть $A=A(m,d,w)$ --- алгебра из предыдущего параграфа, где $m\ge 2, d\le m-1$, $w$ --- 
слово Штурма или бесконечное периодическое слово. Тогда
$$
l_n(A^\#)\le (n+1)^{2d+9}
$$
для всех достаточно больших $n$.
\end{lemma}

\begin{proof}

По лемме \ref{L7}
$$
\dim W_n^{(d+3)}(A) \le d(d+3)(m+1)Comp_w(n).
$$

Так как сложность периодического слова --- константа, а у слова Штурма она равна $n+1$, то
$$
\dim W_n^{(d+3)}(A) \le n^2
$$
для всех достаточно больших $n$. Поэтому
$$
\dim W_n^{(d+3)}(A^\#) \le (n+1)^{d+6}
$$
по лемме  \ref{LL1}. Из замечания \ref{r1} вытекает, что
$$
\chi_n(A^\#)=\sum_{{\lambda\vdash n\atop h(\lambda)\le d+2}} m_\lambda\chi_\lambda,
$$
а $m_\lambda\le \dim W_n^{(d+3)}(A^\#) \le (n+1)^{d+6}$. И поскольку число разбиений
$\lambda\vdash n$ с $h(\lambda) \le d+3$ не превосходит $(n+1)^{d+3}$, то
$$
l_n(A^\#) \le (n+1)^{2d+9}.
$$ 
 \end{proof}

Лемма \ref{LL2} потребуется нам для верхней оценки PI-экспоненты алгебры $A(m,d,w)^\#$.
Но сначала мы оценим рост ее коразмерностей снизу.

\begin{lemma}\label{LL3}
Пусть $A=A(m,d,w)$ задана параметрами $m\ge 2, d\le m-1$ и $w$. Тогда
$$
\underline{exp}(A^\#) \ge exp(A)+1.
$$
\end{lemma}

\begin{proof}
При доказательстве леммы \ref{A2} для любого $\delta >0$ была выбрана возрастающая последовательность $n=n(t), t=t_0,t_0+1,\ldots$, семейство разбиений 
$\lambda^{(n)}\vdash n(t)$ и набор полиномов $f_t, t\ge t_0$, со следующими совйствами:
\begin{itemize}
\item
разбиение $\lambda$ имеет вид $\lambda=(\lambda_1,\ldots,\lambda_{d+2})$, 
$\lambda_1=\ldots= \lambda_d=m_1+\cdots+m_t-t, \lambda_{d+1}=td, \lambda_{d+2}=1$,
\item
$\Phi(\lambda^{(n)})\ge \Phi_d(\frac{1}{m+\alpha}-\delta)$, где $\alpha$ --- наклон $w$,
\item
$n(t+1)-n(t) \le d(m+1)$ для всех $t\ge t_0$,
\item
симметризация $f_t$ не является тождеством $A$  и порождает неприводимый $F[S_n]$-модуль
с характером $\chi_\lambda$,
\item
$f_t$ кососимметричен по $\lambda_1$ наборам переменных: один размера $d+2$, $td-1$ --- размера
$d+1$ и $\lambda_1-\lambda_{d+1}$ --- размера $d$.
\end{itemize}
Кроме того, $exp(A)=\Phi_d(\frac{1}{m+\alpha})$.

Обозначим через $\widetilde h_{t,k}$ произведение
$$
\widetilde h_{t,k}=f_tz_1\ldots z_k, \, k \ge 1.
$$
Рассмотрим ту же подстановку $\varphi$, которая давала ненулевое значенин для $f_t$
и $Sym\, f_t$ и расширим ее действие на $\widetilde h_{t,k}$, положив
$\varphi(z_1)=\ldots=\varphi(z_k)=1$. Тогда, очевидно,
$$
\varphi(\widetilde h_{t,k})=\varphi(f_t)\ne 0.
$$
Более того, если $k\le td$, то мы можем включить $z_1,\ldots,z_k$ в первые $k$ кососимметричных
набора у $_t$ и провести дополнительное альтернирование по расширенным наборам. При этом из 
правил умножения базисных элементов $A$ следует, что
$$
\varphi(Alt(\widetilde h_{t,k}))=\gamma\varphi(\widetilde h_{t,k}),
$$
где $\gamma$ --- ненулевой целочисленный коэффициент. У полинома $f_{t,k}=
Alt(\widetilde h_{t,k})$ переменнные тоже распределены по $\lambda_1$ кососимметричным
наборам: один размера $d+3$, $k-1$ --- размера $d+2$ и $\lambda_1-k$ --- размера $d+1$.
Более того, если провести его симметризацию по тем же переменным, что и для $f_t$ плюс
симметризацию по $z_1,\ldots,z_k$, то значение $\varphi(Sym\,(f_{t,k})$ тоже пропорциоонально
$\varphi(f_{t})$  с ненулевым коэффициентом. То есть полином $Sym\,(f_{t,k}$ порождает
неприводимый $F[S_{n+k}]$-модуль с характером $\chi_\mu$, где
$$
\mu = (\mu _1,\ldots,\mu_{d+3}), \mu_1=\lambda_1,\ldots, \mu_d=\lambda_d,
\mu_{d+1}=\lambda_{d+1}, \mu_{d+2}=k, \mu_{d+3}=1.
$$
Аналогично доказывается, что все разбиения вида
$$
\mu=(\lambda_1,\ldots,\lambda_d,k,\lambda_{d+1},1),\,
\mu=(k,\lambda_1,\ldots,\lambda_{d+2})
$$
имеют ненулевые кратности в характере $\chi_{n+k}(A^\#)$. Другими словами, мы можем добавить 
к диаграмме $D_\lambda$ любую строку (1-ю, $d+1$-ю либо $d+2$-ю) и получить диаграмму 
$D_\mu$, соответствующую разбиению $\mu\vdash n+k$ с ненулевой кратностью.

Оценим снизу максимальное значение $\Phi(\mu)$ и $k$. Обозначим 
$$\frac{\lambda_1}{n}=u_1,
\ldots,\frac{\lambda_{d+2}}{n}=u_{d+2}, \beta=\Phi(\lambda).$$
 Тогда по лемме \ref{L3}
\begin{equation}\label{n0}
\Phi\left(\theta u_1,\ldots,\theta u_{d+2}.1-\theta\right)=1+\Phi(\lambda)
\end{equation}
--- максимальное значение, которое может принимать $\Phi(\mu)$, где $\theta=
\frac{\beta}{\beta+1}$. Это означает, что если $k$ удовлетворяет двум неравенствам
\begin{equation}\label{n1}
\frac{k}{k+1}\le 1-\theta=\frac{1}{\beta+1}\le\frac{k+1}{n+k+1},
\end{equation}
то максимум $\Phi(\mu)$ достигается либо при этом $k$, либо при $k+1$. Соотношение (\ref{n1})
\begin{equation}\label{n2}
\frac{n}{\beta}-1\le k \le \frac{n}{\beta}.
\end{equation}

Напомним, что $n$ и $k$ зависят от $t: n=n(t), k=k(t)$. Учитывая (\ref{n2}) и выбор $n(t)$, мы получаем
\begin{equation}\label{n3}
n(t+1)+k(t+1)-n(t)-k(t) \le\frac{\beta}{\beta+1} d(m+1)+2.
\end{equation}

Обозначим $r=r(t)=n(t)+k(t)$, а через $\mu^{(r)}$ --- разбиение $r(t)$ с максимальным
значением $\Phi(\mu^{(r)})$. так как с ростом $n$ величину $\frac{1}{\beta+1}$ все более
точно аппроксимировать дробью $\frac{k}{n+k}$, то можно с учетом (\ref{n0}) считать, что
$$
\Phi(\mu^{(r)}) \ge \Phi(\lambda^{(n)})+1-\delta'
$$
при всех достаточно больших $n$, где $\delta' >0$ --- любая заранее заданная величина,
$n=n(t), r=r(t)$. Тогда с учетом леммы \ref{L1} мы имеем
\begin{equation}\label{n4}
c_{r(t)}(A^\#) \ge \frac{\Phi\left(\mu^{(r(t))} \right)^n }{n^{(d+2)^2+d+3}} 
\ge  \frac{\left(\Phi(\lambda^{(n)})+1-\delta' \right)^n}{n^{(d+2)^2+d+3}} \ge 
\frac{\left(\Phi_d(\frac{1}{m+\alpha}-\delta)+1-\delta' \right)^n}{n^{(d+2)^2+d+3}}.
\end{equation}

Поскольку все разности $r(t+1)-r(t)$ ограничены общей константой (см. (\ref{n4})), а
последовательность $\{c_n(A^\#)\}$ --- неубывающая, то из (\ref{n4}) следует, что
$$
\underline{\lim}_{n\to\infty}\sqrt[n]{c_n(A^\#)} \ge 
\Phi_d(\frac{1}{m+\alpha}-\delta)+1-\delta'.
$$
Наконец, так как $\delta$ и $\delta'$ --- произвольные сколь угодно малые величины, мы полычаем
$$
\underline{exp}(A^\#) \ge exp(A)+1,
$$
и лемма доказана.

\end{proof}

\subsection{}\label{s4.2}
Теперь мы получим верхнюю оценку на $\underline{exp}(A)$.

\begin{lemma}\label{LL4}
$$
\underline{exp}(A^\#) \le exp(A)+1.
$$
\end{lemma}
 \begin{proof}

Так как кодлина $l_n(A^\#)$ полиномиально ограничена согласно лемме \ref{LL2}, то
достаточно доказать, что
$$
\Phi(\lambda) \le \Phi_d(\frac{1}{m+\alpha})+1=exp(A)+1,
$$
для любого $\lambda\vdash n$ с $m_\lambda\ne 0$ в $\chi_n(A)$ как показывает соотношение
(\ref{e2a}).

Пусть $h=h(x_1,\ldots,x_n)$ --- полилинейный многочлен, не являющийся тождеством $A^\#$,
порождающий в $P_n$ неприводимый $F[S_n]$-модуль с характером $\chi_\lambda$. Как отмечалось
ранее, можно считать $h$ кососимметричным по $\lambda_1$ наборам переменных, причем
$\lambda_{d+2}$ из них имеют размер не меньше $\lambda_{d+2}$. Если $\lambda_{d+2}=0$, то
$$
\Phi(\lambda)\le \Phi(\underbrace{\frac{1}{d+1},\ldots,\frac{1}{d+1}}_{d+1},0,0)
= d+1 <1+\Phi_d(\frac{1}{m+\alpha}).
$$

Пусть $\lambda_{d+2}\ne 0$. Зафиксируем произвольное $\varepsilon >0$. Поскольку
$h\not\in Id(A^\#)$, то существует подстановка $\varphi$ базисных элементов $A$ и 1 вместо
переменных $x_1,\ldots,x_n$, при которой
$$
\varphi(h)=f(z^i_{jk},a_1,\ldots, a_d,b)=f
$$
--- ненулевой одночлен степени $n'$ от $\{z^i_{jk},a_1,\ldots, a_d,b \} $, где $n'=n-n_1$, а
$n_1$ --- количество единиц из $A^\#$, подставленных вместо $x_1,\ldots,x_n$. Из структуры
$h$ следует, что $n_1 \ge\lambda_{d+2}$ и $\deg_b f \ge\lambda_{d+2}$. Если
$\lambda_{d+1}=\lambda_{d+2}$, то обозначим через $D_\mu$ диаграмму Юнга, полученную из $D_\lambda$ вычеркиванием $(d+1)$-й строки. Тогда $\mu$ --- разбиение числа
$n''=n-\lambda_{d+2}$ и $\mu_{d+2}\le 1, \mu_{d+1}=\lambda_{d+2}\ge \deg_b f= n'=n-n_1 
\le n-\lambda_{d+2}$. Тогда по лемме \ref{A0}
\begin{equation}\label{n5}
\frac{m_{d+2}}{n''} = \frac{m_{d+1}}{n-\lambda_{d+2}} \le \frac{\deg_b f}{n-\lambda_{d+2}}
\le \frac{\deg_b f}{n'} \le \frac{1}{m+\alpha}+\varepsilon_{n'}.
\end{equation}

Теперь покажем, что соотношение, аналогичное (\ref{n5}) можно получить и при 
$\lambda_{d+1}>\lambda_{d+2}$. В этом случае в полиноме $h$ кроме $\lambda_{d+2}$
кососимметричных наборов порядка $d+2$ есть еще $\lambda_{d+1}-\lambda_{d+2}$
кососимметричных наборов порядка $d+1$. В каждый из этих наборов подставлен один из элементов
$\{1,b\}$, поскольку $\varphi(h)\ne 0$. Пусть $b$ подставлен ровно в $k\le \lambda_{d+1}-\lambda_{d+2}$ из этих наборов. Тогда $deg_b f\ge r+k$, где $r=\lambda_{d+2}$.
Обозначим также $\lambda_{d+1}-k=t$. Тогда в $D_\lambda$ есть две строки длин
$t+k$ и $r$, причем $t\ge r$. Перебрасывая клетки из $(d+1)$-й строки $D_\lambda$ в 
$(d+2)$-ю  можно получить диаграмму $D_{\lambda'}$, у которой длины строк
$d+1,d+2$ равны $r+k$ и $t$. По лемме \ref{L2} мы имеем: $\Phi(\lambda') \ge 
\Phi(\lambda)$. Теперь через $D_\mu$ мы обозначим диаграмму $D_{\lambda'}$ с
вычеркнутой $(d+2)$-й строкой. Тогда снова $\mu_{d+1}=r+k\le \deg_b f$,
$\mu\vdash n'' \ge n'$, и мы снова получаем соотношение
\begin{equation}\label{n6}
\frac{\mu_{d+1}}{n''} \le \frac{\deg_b f}{n'} \le \frac{1}{m+\alpha}+\varepsilon_{n'}.
\end{equation}

Заметим сначала, что $\lambda_1 \ge n_1$ в силу кососимметричности $h$ по $\lambda_1$
наборам переменных. Обозначим $x=\lambda_1$. Тогда
$$
\Phi(\lambda) \le 
\Phi\left(x,\underbrace{\frac{1-x}{d+2},\ldots,\frac{1-x}{d+2} }_{d+2}\right)=H(x)
$$Предел функции $H(x)$ при $x\to\infty$ равен $1$. Это, в частности, означает, что
существует такое целое $q$, что если $n_1\ge\frac{q-1}{q}n$, то $\Phi(\lambda) <d$ для всех
достаточно больших $n\ge N$.

Разделим теперь все разбиения $\lambda\vdash n\ge N$ на две группы --- где 
$\lambda_1 >\frac{q-1}{q}N$ и где $\lambda_1 \le\frac{q-1}{q}N$. Для всех разбиений
первой группы неравенство
$$
\Phi(\lambda) < d < \Phi_d(\frac{1}{m+\alpha}+\varepsilon)
$$
выполняется в силу выбора $q$ и $n$. Для разбиений из второй группы воспользуемся соотношениями
(\ref{n5}) и (\ref{n6}). В первом случае диаграмма $D_\mu$ получена из $D_\lambda$
вычеркиванием одной строки. Поэтому по лемме \ref{L3} мы имеем 
$\Phi(\lambda)\le \Phi(\mu)+1$. Во втором случае $D_\mu$ получена из $D_{\lambda'}$ вычеркиванием строки, а $D_{\lambda'}$ получена из $D_{\lambda}$ переносом вниз
нескольких клеток. Поэтому по леммам \ref{L2} и \ref{L3}
$\Phi(\lambda) \le \Phi(\lambda') \le \Phi(\mu)+1$. В любом из случаев из \ref{n5},
\ref{n6} получаем
$$
\Phi(\lambda) \le  \Phi(\mu)+1 \le 
\Phi(\theta,\ldots,\theta,\frac{1}{m+\alpha}+\varepsilon_{n'},\frac{1}{n}),
$$
используя свойства $\Phi$, где
$$
(d+1)\theta+\frac{1}{m+\alpha}+\varepsilon_{n'}+\frac{1}{n}=1.
$$

Так как $\lambda_1 \le\frac{q-1}{q}n$, то $n'=n-n_1\ge n-\lambda_1\ge \frac{n}{q}$.
Поэтому $n'\to\infty$ с ростом $n$ и $\varepsilon_{n'}\to 0$. Как и в доказательстве леммы
\ref{A0}, получаем, что
$$
\Phi(\lambda) \le 1+\Phi(\theta',\ldots,\theta',\frac{1}{m+\alpha}+\varepsilon,0,0)=
1+\Phi_d(\frac{1}{m+\alpha}+\varepsilon)
$$
для всех достаточно больших $n$. Поскольку $\varepsilon >0$ выбрано произвольно, мы получаем
$$
\overline{exp}(A^\#) \le 1+\Phi_d(\frac{1}{m+\alpha})= exp(A)+1.
$$

 \end{proof}

Комбинация лемм \ref{LL3} и \ref{LL4} сразу дает следующий результат.

\begin{theorem}\label{t2}
Пусть $m$ и $d$ --- целые числа, $m\ge 2, m-1\ge d$, а $w$ --- бесконечное периодическое слово
или слово Штурма. Если $A=A(m,d,w)$ и $A^\#$ получена из $A$ присоединением  единицы, то
$exp(A^\#)$ существует, причем $exp(A^\#)=exp(A)+1$
\end{theorem}

\begin{corollary}\label{c1}
Для любого вещественного числа $\gamma\ge 2$ существует алгебра $A_\gamma$ с
единицей с PI-экспонентой $exp(A_\gamma)=\gamma$.
\end{corollary}

\begin{proof}
При заданном $d$ совокупность значений
$$
\left\lbrace \Phi_d(\frac{1}{m+\alpha})=exp(A(m,d,w)) \vert 0\le\alpha\le 1,
\quad m=d+1, d+2,\ldots\right\rbrace
$$
покрывает весь промежуток $(d,d+1]$. Следовательно, любое вещественное число $\gamma>2$
реализуется как экспонента $exp(A^\#)$, где $A=A(m,d,w)$ для подходящих $m,d$ и $w$.
Для $\gamma=2$ есть много реализаций даже в ассоциативном случае. Например, для бесконечномерной алгебры Грассмана $G$ с единицей $c_n(G)=2^{n-1}$ (\cite{KR} или
 \cite[теорема 4.1.8]{GZbook}). Поэтому $exp(G)=2$.

\end{proof}

Отдельный интерес представляет вопрос о множестве значений PI-экспонет конечномерных
алгебр. Ясно, что если поле $F$ счетно, то и это множество счетно. В работе \cite{GMZ}
показано, что множество $\{exp(A)|\dim A <\infty\}$ всюду плотно в $[1;\infty)$, а в работе
\cite{Z} доказано, что для конечномерной унитарной алгебры $A$ рост $\{c_n(A)\}$ либо
полиномиален, либо ограничен снизу показательной функцией $2^n$

Еще одним следствием теорем \ref{t1} и \ref{t2} является тот факт, что совокупность 
PI-экспонент конечномерных алгебр с единицей является всюду плотным подмножеством в области
 $[2;\infty)\subset \mathbb{R}$.
 
\begin{corollary}\label{c2}
Для любых вещественных $2\le\alpha <\beta$ существует конечномерная алгебра $B$, такакя, что
$$
\alpha \le exp(B) \le \beta.
$$
\end{corollary}

\begin{proof}
Рассмотрим алгебру $A(m,d,w)$, где $w$ --- бесконечное периодическое слово с периодом $T$,
и вместе с ней --- конечномерную алгебру $B=B(m,d,w)$  с базисом
$$
\left\lbrace a_1,\ldots,a_d,b, z^i_{jk}\vert
1\le i \le d,\,  1\le j \le m+w_j,\, 1\le k \le T \right\rbrace
$$
и таблицей умножения
$$
z^i_{jk} a_i=
\left\{
  \begin{array}{rcl}
     z^i_{j+1,k}, &\quad \hbox{если} \quad & j<m+w_k  \\
    0, &\quad \hbox{если} \quad & j=m+w_k\, ,
           \end{array}
\right.
$$ 
$$
z^i_{m+w_k,k} b=
\left\{
  \begin{array}{rcl}
     z^{i+1}_{1k}, &\quad \hbox{если} \quad & i<d  \\
    z_{1,k+1}^1, &\quad \hbox{если} \quad & i=d,\, k<T \\
    z_{11}^1, &\quad \hbox{если} \quad & i=d,\, k=T\, .
           \end{array}
\right.
$$ 
Легко заметить, что алгебры $A(m,d,w)$ и $B(m,d,w)$ PI-эквивалентны, т.е. имеют одни и те 
же тождества. Но тогда и алгебры $A(m,d,w)^\#$ и $B(m,d,w)^\#$ тоже PI-эквивалентны. В частности, $exp(A(m,d,w)^\#)=exp(B(m,d,w)^\#)$. В частности, $exp(B(m,d,w)^\#) = 
exp(A(m,d,w))+1$.

По предложению \ref{p1} для любого рационального $q\in(0;1)$ существует периодическое слово 
$w$ с наклоном $\pi(w)=q$. Но тогда
$$
exp(B(m,d,w)^\#)=\Phi_d(\frac{1}{m+q})+1
$$
в силу теоремы \ref{t2}. Поэтому можно подобрать такое рациональное положительное $q < 1$,
что $\alpha \le exp(B(m,d,w)^\#) \le \beta$.

\end{proof}

\end{fulltext}


\end{document}